\documentclass[11pt]{llncs}
\usepackage{amssymb}
\usepackage{amsmath}
\usepackage{amsfonts}
  \usepackage{geometry}
  \newgeometry{vmargin={30mm}, hmargin={30mm}} 
\usepackage{graphicx}
\usepackage{graphics}
\usepackage{color}
\usepackage{url}
\usepackage{wasysym}
\usepackage{newlfont}
\usepackage{enumerate}
\usepackage{stmaryrd}
\newcommand{\jumptmp}[2]{#1\llbracket{#2}#1\rrbracket}
\newcommand{\jump}[1]{\jumptmp{}{#1}}
\usepackage{xspace}
\newcommand{\gismo}{{\fontfamily{phv}\fontshape{sc}\selectfont 
G\pmb{+}Smo}\xspace}
\newcommand{\normal}{\mathbf{n}}

\newcommand{\thmref}[1]{Theorem~\ref{#1}}
\newcommand{\lemref}[1]{Lemma~\ref{#1}}
\newcommand{\propref}[1]{Proposition~\ref{#1}}

\numberwithin{equation}{section}
    \urldef{\mailsa}\path|moorekwesi@gmail.com|
    \newcommand{\keywords}[1]{\par\addvspace\baselineskip  
     \noindent\keywordname\enspace\ignorespaces#1}
\begin{document}
\mainmatter 
\title{Discontinuous Galerkin Isogeometric Analysis for Elliptic Problems with 
Discontinuous Coefficients on Surfaces}
\titlerunning{dGIGA for Ellptic Problems on Surfaces}
\author{Stephen Edward Moore}
\authorrunning{S.~E.~Moore }
\institute{Department of Mathematics, 
\footnote{Corresponding address: stephen.moore@ucc.edu.gh},\\ 
University of Cape Coast,\\ Cape Coast, Ghana.\\
}
\maketitle
\begin{abstract}
This paper is concerned with using discontinuous Galerkin isogeometric analysis 
(dGIGA) as a numerical treatment of Diffusion problems on orientable surfaces 
$\Omega \subset \mathbb{R}^3$. The computational domain or surface considered 
consist of several non-overlapping sub-domains or patches which are coupled via 
an interior penalty scheme. In Langer and Moore \cite{LangerMoore:2014a}, we 
presented \textit{a priori} error estimate for conforming computational domains 
with matching meshes across patch interface and a constant diffusion 
coefficient. However, in this article, we generalize the \textit{a priori} error 
estimate to non-matching meshes and discontinuous diffusion coefficients across 
patch interfaces commonly occurring in industry. We construct B-Spline or NURBS 
approximation spaces which are discontinuous across patch interfaces. We present 
\textit{a priori} error estimate for the symmetric discontinuous Galerkin scheme 
and numerical experiments to confirm the theory.
\end{abstract}
\,\,
\section{Introduction}
\label{sec:Introduction}
In this paper, we consider the second-order elliptic boundary value problem on a 
open, smooth, connected and oriented two dimensional surface $\Omega \subset \mathbb{R}^3$ as follows: 
find $u : \overline{\Omega} \rightarrow \mathbb{R}$ such that
\begin{align}
 \label{eqn:modelproblem}
     -\text{div}_\Omega(\alpha \nabla_\Omega u ) + u= f \,\, \text{in} \,\, 
\Omega, 
    \quad  u  = 0   \,\, \text{on} \,\, \Gamma_D, 
    \quad  \normal \cdot (\alpha \nabla_\Omega u )  = g_N   \,\, \text{on} \,\,  
\Gamma_N,
\end{align}
where the diffusion coefficient $\alpha$ is uniformly bounded i.e. $\alpha_{min} 
\leq \alpha \leq \alpha_{max}$ with positive constants $\alpha_{max}$ and 
$\alpha_{min}$, $f$ and $g_N$ are given sufficiently smooth data. The physical 
or computational domain $\Omega \subset \mathbb{R}^3$ is compact, connected and 
positively oriented surface with boundary $\partial \Omega.$ The boundary of the 
computational domain consists of the Dirichlet part $\Gamma_D$ with positive 
boundary measure and a Neumann part $\Gamma_N$ such that $\partial \Omega := 
\Gamma_D \bigcup \Gamma_N.$ The operators $\text{div}_\Omega$ and 
$\nabla_\Omega$ are the surface divergence and surface gradient respectively, 
and will be defined in Section~\ref{sec:preliminaries}. 

Partial Differential Equations (PDEs) on surfaces arise in many fields of 
application like 
material science, fluid mechanics, electromagnetics, biology and image 
processing, see e.g.\cite{DziukElliott:2013a} for several interesting 
discussions on applications.
For several years, numerical methods dedicated to the solutions of PDEs on 
manifolds including conforming and non-conforming finite element methods (FEM) 
have been well studied and applied to compute the solution of elliptic and 
parabolic evolution problems on fixed and evolving computational domains, see, 
e.g.,
 \cite{DziukElliott:2013a,DednerMadhavanStinner:2013a}. 
 We note that there are however some drawbacks to the standard surface FEM. The 
standard surface FEM has two main sources of error: the error due to the 
approximation of the infinite dimensional spaces with finite dimensional spaces 
in the variational problem and the geometric error resulting from the 
approximation of the surface. These drawbacks are due to the discrete 
variational formulation of the PDE that is constructed on a triangulated surface 
which contains the finite elements space as discussed by Dzuik and Elliott in 
\cite{DziukElliott:2013a}.
 
As an alternative aproach to the surface FEM, we resort to Isogeometric Analysis 
(IGA).
The numerical scheme is based on B-splines and Non-Uniform B-splines (NURBS) and 
was proposed to approximate solutions of PDEs, see e.g. 
\cite{HughesCottrellBazilevs:2005a}. The method uses the same class of basis 
functions for representing both the geometry of the computational domain and 
also approximating the solution of problems modeled by PDEs. 
By using the exact representation of the geometry, the geometrical errors 
introduced by approximation of computational domains in the surface FEM are 
eliminated. 
This is especially of importance in the discretization of PDEs on surfaces.
However, we note that IGA can also have geometry-related failures such as holes, 
singularities, etc. see e.g.\cite{WuWangMourrainNkongaCheng:2017a}. 
Such failures or features are beyond the scope of this article.  
IGA uses B-splines or Non-Uniform Rational B-Splines (NURBS) basis functions 
which are standard in Computer Aided Design (CAD). The NURBS basis functions 
have several advantages making them suitable for analysis, see 
\cite{HughesCottrellBazilevs:2005a}. 
The mathematical analysis of the approximation properties, the stability and 
discretization error estimates of NURBS spaces and analysis of several 
refinement strategies, i.e., $h$-$p$-$k$ refinements can be found in 
\cite{BazilevsBeiraoCottrellHughesSangalli:2006a}.
In many practical applications, the computational domains cannot be represented 
by a single B-spline or NURBS domain but by several patches or sub-domains. 
In this sense, single patch IGA and multi-patch IGA have been addressed in 
\cite{DaVeigaBuffaSangalliVazquez:2014a}. 

Alternatively, multi-patches can also be coupled via interior penalty Galerkin 
methods. In our earlier articles, see e.g., 
\cite{LangerMoore:2014a,LangerMantzaflarisMooreToulopoulos:2014a,Moore:2018a}, 
we analyzed the multi-patch discontinuous Galerkin IGA (dGIGA) for diffusion and 
biharmonic problems and presented several convincing numerical results for 
conforming domains with matching meshes. 
However, in this paper, we will generalize the analysis to include non-matching 
meshes with jumping diffusion coefficients across patch boundaries and present 
\textit{a priori} error estimates for diffusion problems. Our analysis follows 
the monograph \cite{DiPietroErn:2012a} and requires three main ingredients; 
discrete stability, consistency and boundedness of the discrete bilinear form. 
Then using the approximation estimates, see e.g., 
\cite{BazilevsBeiraoCottrellHughesSangalli:2006a}, we finally derive \textit{a 
priori} error estimate. The linear system obtained from the discretization of 
the problem is solved by means of a preconditioned conjugate gradient (PCG) with 
a scaled Dirichlet preconditioner as presented in primal isogeometric tearing 
and interconnecting (dG-IETI-DP) see e.g., \cite{Hofer:2018}.

The rest of the paper is organized as follows; Section~\ref{sec:preliminaries} 
gives a brief introduction to function spaces, weak formulation, NURBS surfaces 
and geometrical mappings and isogeometric analysis. We present the dGIGA scheme 
in Section~\ref{sec:dGIGASchemeFormulation}. In 
Section~\ref{sec:AnalysisOfDiscretizationError}, we present the multi-patch 
dGIGA  and the analysis of the dGIGA scheme. The \textit{a priori} error 
estimate is presented in Section~\ref{sec:erroranalysisOfdgigadiscretization}. 
We present numerical results for an open surface and a closed surface with 
non-matching meshes 
respect to the jumping diffusion coefficient 
in Section~\ref{sec:NumericalResults}. Finally, we conclude and give an outlook.
\section{Preliminaries}
\label{sec:preliminaries}
In this section, we introduce briefly introduce Sobolev spaces, NURBS surfaces 
and isogeometric analysis method, see e.g. 
\cite{AdamsFournier:2008,HughesCottrellBazilevs:2005a} for detailed study.
Firstly, let the computational domain $\Omega$ be a compact smooth and oriented 
surface 
with boundary $\partial \Omega$.  We introduce the Sobolev space 
$H^s(\Omega) :=\{v \in L_2(\Omega) \, : \, D^{\kappa} v \in L_2(\Omega), \, \, 
\text{for} \, \, 0 \leq |\kappa| \leq s  \}$, where $L_2(\Omega)$ denote the 
space of square integrable functions and $\kappa = (\kappa_1,\ldots,\kappa_d)$ 
be a multi-index with non-negative integers $\kappa_1,\ldots,\alpha_d$, and  
$|\kappa| = \kappa_1 +\ldots + \kappa_d,$ $D^{\kappa} := 
\partial^{|\kappa|}/\partial x^\kappa.$ We associate the Sobolev space 
$H^s(\Omega)$ with the norm 
 $\|v\|_{H^s(\Omega)} = \left( \sum_{0 \leq |\kappa| \leq s} \|D^\kappa 
v\|^2_{L_2(\Omega)} \right)^{1/2}.$ 

The variational formulation of the surface diffusion problem 
\eqref{eqn:modelproblem} reads:
find $u \in V_0$ such that 
\begin{align}
\label{eqn:variationalformulation}
     a(u,v) & = \ell(v), \quad \forall v \in V_0,
\end{align}
where the bilinear and linear forms are given by 
\begin{align}
 a(u,v) = \int_{\Omega} \alpha \nabla_\Omega u \cdot \nabla_\Omega v + uv \,\, 
dx \quad\text{and} \quad 
 \ell(v) = \int_\Omega f v\, dx + \int_{\Gamma_N} g_N v\,ds,
\end{align}
with
$V_0 :=\{v \in H^1(\Omega) : v = 0 \quad  \text{on} \quad \Gamma_D \}.$ 
The existence and uniqueness of the solution of such a variational problem 
\eqref{eqn:variationalformulation} follows the standard arguments of Lax-Milgram 
lemma if $u \in H^2(\Omega)$ satisfies 
\begin{align}
 \label{eqn:laxmilgram}
 \| u \|_{H^2(\Omega)} \leq \|f\|_{L_2(\Omega)},
\end{align}
see e.g. \cite{Wloka:1987a} for further details.
\subsection{NURBS Geometrical Mapping and Surfaces}
\label{subsec:surfaces}
{\allowdisplaybreaks
Let $\xi = (\xi_1,\xi_2) \in \mathbb{R}^2$ be a vector-valued independent 
variable in the parameter domain $\widehat{\Omega}$. By means of a smooth and 
invertible geometrical mapping $\mathrm{\Phi}$, the computational  domain 
$\Omega \subset \mathbb{R}^3$ is defined as
 \begin{equation}
\label{eqn:geometricalmapping}
   \mathrm{\Phi} : \widehat{\Omega} \rightarrow \Omega \subset \mathbb{R}^3, 
  \quad 
  \mathbf{\xi} \rightarrow x = \mathrm{\Phi}(\mathbf{\xi}),
\end{equation}
where the parameter domain $\widehat{\Omega} \subset \mathbb{R}^2$ as 
illustrated in Fig.~\ref{fig:igamap}.

  \begin{figure}[th!]
  \centering
      \includegraphics[width=0.8\textwidth]{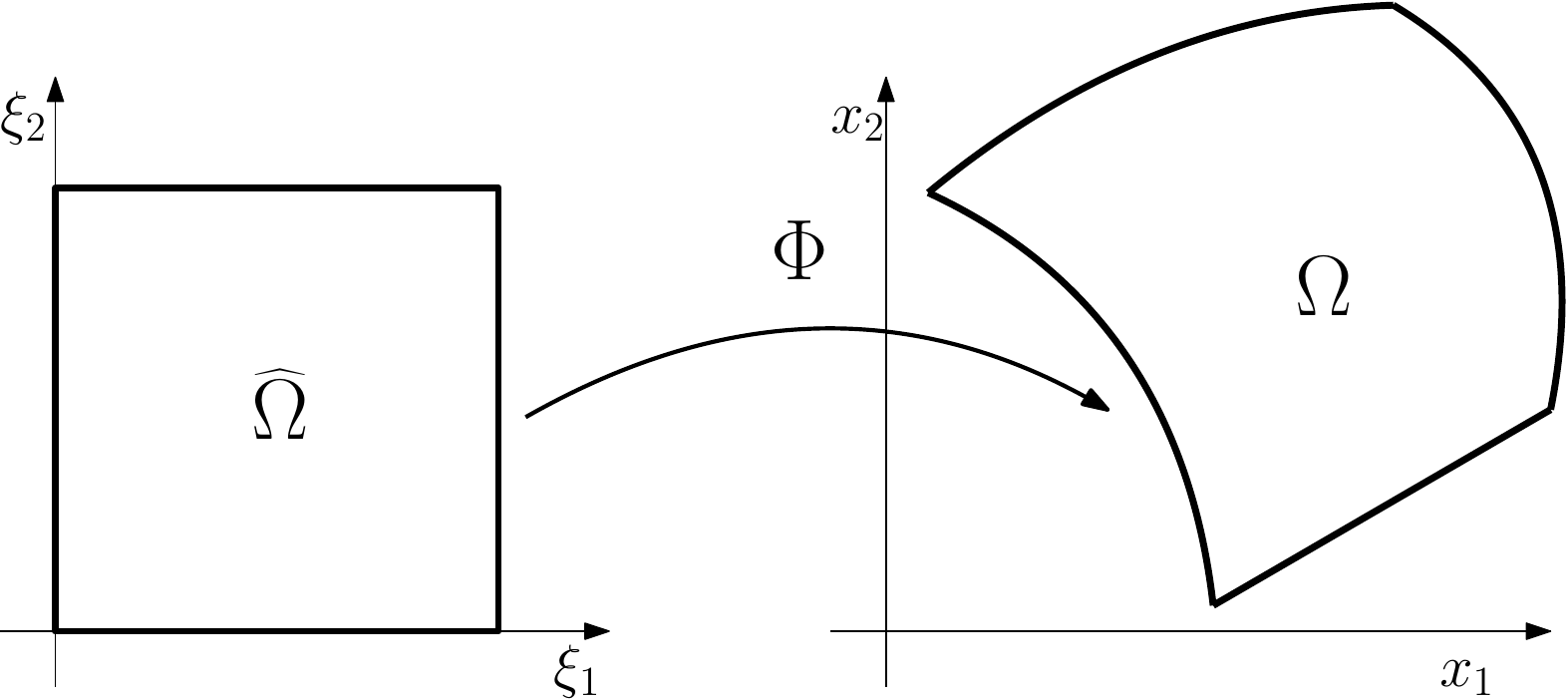}
      \caption{Illustration of the geometrical map $\Phi : \widehat{\Omega} 
\subset \mathbb{R}^2 \rightarrow \Omega \subset \mathbb{R}^3$ for a patch.}
      \label{fig:igamap}
 \end{figure}

We introduce briefly some important mathematical tools necessary for the 
analysis of surface PDEs. The following objects are obtained by means of the 
geometrical mapping \eqref{eqn:geometricalmapping} in the parameter domain. The 
Jacobian $\widehat{J},$ first fundamental form $\widehat{F}$ and the determinant 
$\widehat{g}$ of the geometrical mapping are respectively given by
\begin{align}
 \widehat{J} 
	&:=
	\left[ 
	\frac{\partial \mathrm{\Phi}_k}
 	{\partial \mathbf{\xi}_l}	
	\right] \in \mathbb{R}^{3 \times 2}, \quad k  = 1,2,3, \,l  = 
1,2,\label{eqn:jacobian}\\ 
 \label{eqn:firstfundamentalanddeterminant}
 \widehat{F}(\xi) &= \left(\widehat{J}(\xi) \right)^{T}\widehat{J}(\xi) \in 
\mathbb{R}^{2 \times 2} \quad \text{and} \quad \widehat{g}(\xi)= \sqrt{\det 
\left({\widehat{F}(\xi) }\right)} \in \mathbb{R}.
\end{align}

Next, we present some differential operators by using notations in the parameter 
domain.
We consider a smooth function $\phi$ defined on the manifold $\Omega,$ by using 
the invertible geometrical mapping \eqref{eqn:geometricalmapping} to obtain   
\begin{equation}
 \label{eqn:invertiblemap}
 \phi ( x) = \widehat{\phi}(\xi) \circ \Phi^{-1}(x), \quad x \in \Omega,
\end{equation}
where $\widehat{\phi}(\xi) = \phi (\mathrm{\Phi}(\xi)).$
Using the gradient operator in the parameter space $\nabla \widehat{\phi}$, the 
tangential gradient of the manifold is given by 
\begin{equation}
\label{eqn:tangentialgradient}
  \nabla_{\Omega}\phi (x):= 
\widehat{J}(\mathbf{\xi})\widehat{F}^{-1}(\mathbf{\xi}) \nabla 
\widehat{\phi}(\mathbf{\xi}) \circ \Phi^{-1}(x).
\end{equation}
The divergence operator for the vector-valued function can be written as 
\begin{equation}
\label{eqn:divergenceoperator}
  \nabla_{\Omega} \cdot \phi (x):= \dfrac{1}{\widehat{g}\mathbf{\xi})} \nabla 
\cdot \left[\widehat{g}(\mathbf{\xi}) \widehat{F}^{-1}(\xi)\widehat{J}^T(\xi) 
\widehat{\phi}(\mathbf{\xi}) \right] \circ \Phi^{-1}(x) .
\end{equation}
The Laplace-Beltrami operator on the manifold $\Omega$ is defined for a twice 
continuously differentiable function $\phi :\Omega \rightarrow \mathbb{R}$ as 
\begin{equation}
\label{eqn:laplacebeltrami}
  \Delta_{\Omega}\phi (x) =  \dfrac{1}{\widehat{g}(\mathbf{\xi})} \nabla \cdot 
\left[\widehat{g}(\mathbf{\xi}) \widehat{F}^{-1}(\mathbf{\xi}) 
  \nabla \widehat{\phi}(\mathbf{\xi})\right] \circ \Phi^{-1}(x) .
\end{equation}
The unit normal vector on the manifold $\Omega \subset \mathbb{R}^3$ is obtained 
by the geometrical mapping of 
\begin{equation}
 \label{eqn:unitnormalvector}
  \widehat{\normal}(\xi) := \dfrac{\widehat{t}_1(\xi) \times 
\widehat{t}_2(\xi)}{\| \widehat{t}_1(\xi) \times \widehat{t}_2(\xi)\|},
\end{equation}
where $\widehat{t}_l(\xi):= \partial \Phi(\xi) / \partial \xi_l$ is the tangent 
vector to a curve in $\mathbb{R}^3$ with $l = 1,2.$ The manifold $\Omega$ has a 
tangent plane at $\xi$ if the tangent vectors are linearly independent.

Finally, by means of the geometrical mapping \eqref{eqn:geometricalmapping}, we 
can write the Jacobian, first fundamental form and the determinant on the 
computational domain $\Omega$ as follows
\begin{align}
 \label{eqn:differentialoperatormanifold}
  J(x)  = \widehat{J}(\xi) \circ \mathrm{\Phi}^{-1}(x) ,  \quad 
  F(x) = \widehat{F}(\xi) \circ \mathrm{\Phi}^{-1}(x)  \quad \text{and} \quad 
  g(x) = \widehat{g}(\xi) \circ \mathrm{\Phi}^{-1}(x).
\end{align}
}

\subsection{NURBS and Isogeometric Analysis}
\label{subsec:bsplineandnurbs}
We begin by introducing the univariate B-splines since they
are usually the industry standard.
Given positive integers $p$ and $n,$ we define a vector 
$\mathrm{\Xi} := \left \{0=\xi_1,\ldots,\xi_{n+p+1}=1\right\}$ 
with a non-decreasing sequence of real numbers in the unit interval or parameter 
domain $\widehat{\Omega} = [0,1]$ called a knot vector. Given  $\mathrm{\Xi}, p 
\geq 1,$ and $n$ the number of basis functions, the univariate B-spline 
functions are defined by the following recursion formula
\begin{align}
\label{eqn:bsplinebasisfunction}
 \widehat{B}_{i,0}(\xi)  & = \left\{
   \begin{aligned}
     & 1 & \text{if} \quad & \xi_{i} \leq \xi < \xi_{i+1},\\
     & 0 & \text{else}, & \notag \\
   \end{aligned}
    \right. \\
     \widehat{B}_{i,p}(\xi)  & = \frac{\xi-\xi_{i}}{\xi_{i+p}
     - \xi_{i}}\widehat{B}_{i,p-1}(\xi)
     +\frac{\xi_{i+p+1}-\xi}{\xi_{i+p+1}-\xi_{i+1}}\widehat{B}_{i+1,p-1}(\xi),
\end{align}
where a division by zero is defined to be zero. 
We note that a basis function of degree $p$ is $(p-m)$ times continuously 
differentiable across a knot value with the multiplicity $m$. 
If all internal knots have the multiplicity $m = 1$, then B-splines of degree 
$p$ are globally $(p-1)-$continuously differentiable.  

The bivariate B-spline basis functions are tensor products of the univariate 
B-spline basis functions \eqref{eqn:bsplinebasisfunction}. 
Let $ \mathrm{\Xi}_k = \left \{\xi_{1,k},\ldots,\xi_{n_k+p_k+1,k}\right\}$ 
be the knot vectors for every direction $k = 1,2$.
Let $\mathbf{i}:= (i_1,i_2),  \mathbf{p}:= (p_1,p_2)$ and the set 
$\overline{\mathcal{I}}=\{\mathbf{i} = (i_1,i_2) : i_k= 1,2, \ldots, n_k; \;   k 
= 1,2\} $ be multi-indicies.  Then the tensor product B-spline basis functions 
are defined by
\begin{equation}
\label{eqn:multivariatebspline}
   \widehat{B}_{\mathbf{i},\mathbf{p}}(\xi) := 
\prod\limits_{k=1}^{2}\widehat{B}_{i_k,p_k}(\xi_k),
\end{equation}
where $\xi = (\xi_1,\xi_2) \in \widehat{\Omega} = (0,1)^2.$ The univariate and 
bivariate B-spline basis functions are defined in the parametric domain by means 
of the corresponding B-spline basis functions $\{ 
\widehat{B}_{\mathbf{i},\mathbf{p}} \}_{\mathbf{i} \in 
\overline{\mathcal{I}}}.$ 

The distinct values $\xi_l, l=1,\ldots,n$ of the knot vectors $\mathrm{\Xi}$ 
provides a partition of $(0,1)^2$ creating a mesh $\widehat{\mathcal{K}}_h$ in 
the parameter domain where $\widehat{K}$ is a mesh element. 
The computational domain is described by means of a geometrical mapping 
$\mathbf{\Phi}$ such that  $\Omega = \mathbf{\Phi}(\widehat{\Omega})$ and 
\begin{align}
  \mathbf{\Phi}(\xi) := \sum_{\mathbf{i} \in \overline{\mathcal{I}}} 
C_\mathbf{i} \widehat{B}_{\mathbf{i},\mathbf{p}}(\xi),
\end{align}
where $C_\mathbf{i} $ are the control points. 

We define the basis functions in the computational domain by means of the 
geometrical mapping as 
$B_{\mathbf{i},\mathbf{p}} := \widehat{B}_{\mathbf{i},\mathbf{p}} \circ 
\mathbf{\Phi}^{-1}$ and the discrete function space by 
\begin{equation}
 \label{eqn:discretefunctionspace}
 \mathbb{V}_h = \text{span}\{B_{\mathbf{i},\mathbf{p}}: \mathbf{i} \in  
\overline{\mathcal{I}} \}. 
\end{equation}
Finally, the NURBS isogeometric analysis scheme reads as folows; Find $u_h \in 
\mathbb{V}_h$ such that
\begin{equation}
 \label{eqn:discretevariationalform}
 a(u_h,v_h)=\ell(v_h), \quad \forall v_h \in \mathbb{V}_h,
\end{equation}
with $\mathbb{V}_h \subset V_0.$

However, for many practical applications, the physical domain $\Omega$ consists 
of non-overlapping domains $\Omega_i, i=1,\dots,N$ called subdomains denoted by 
$\mathcal{T}_h := \{ \Omega_i \}_{i=1}^N$ such that $\overline{\Omega} = 
\bigcup_{i=1}^N \overline{\Omega}_i$ and $ \Omega_i \cap \Omega_j = \emptyset$ 
for $i \neq j.$ Each patch is the image of an associated geometrical mapping 
$\mathbf{\Phi}_i$ such that $\mathbf{\Phi}_i(\widehat{\Omega})= \Omega_i, 
i=1,\ldots,N,$ see Fig.~\ref{fig:multipatchigamap}.
  \begin{figure}[thb!]
  \centering
      \includegraphics[width=0.8\textwidth]{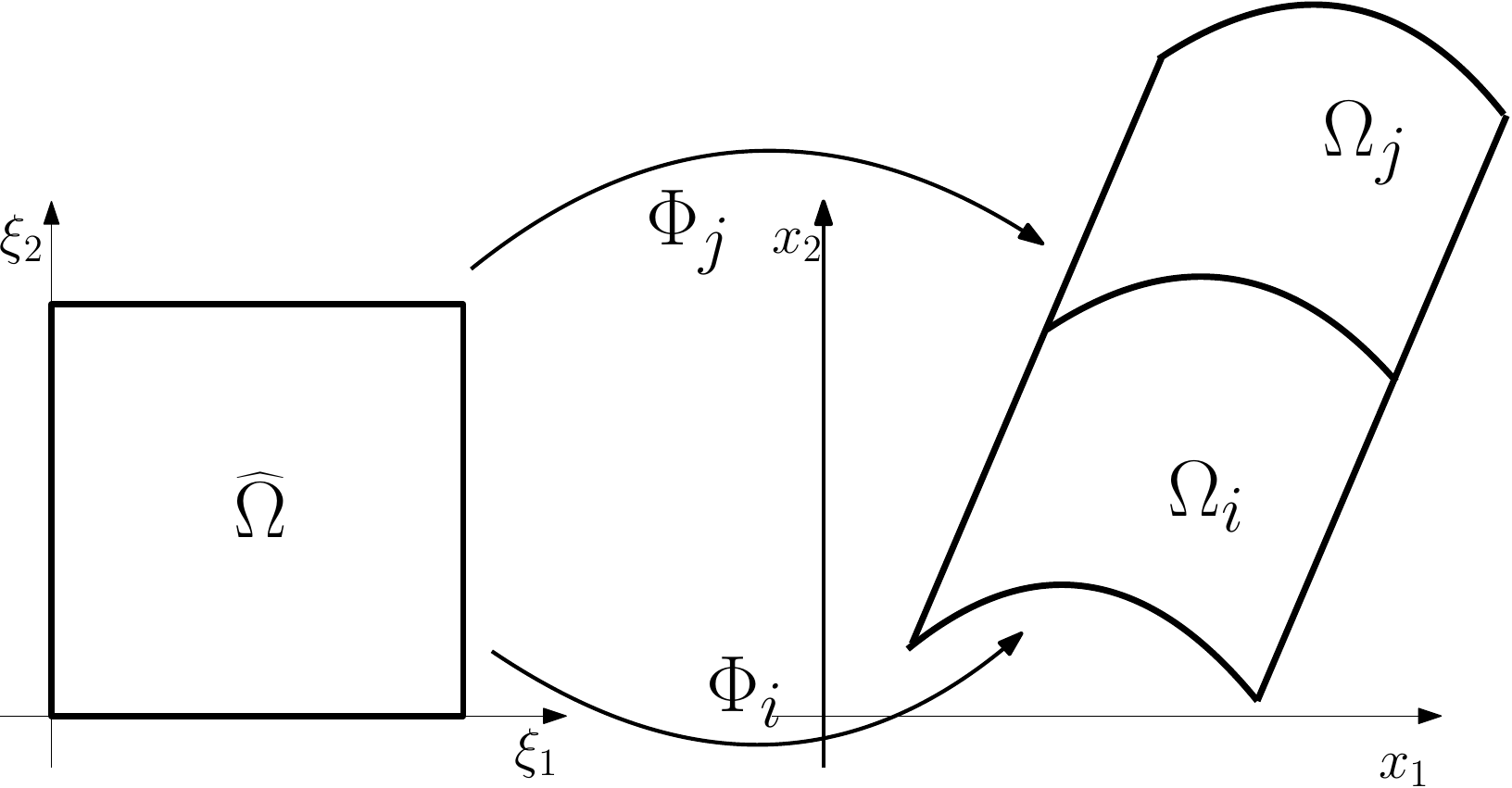}
      \caption{Illustration of the multi-patch isogeometric analysis map $\Phi_i 
(\widehat{\Omega}) =  \Omega_i$ and $\Phi_j (\widehat{\Omega}) =  \Omega_j, i 
\neq j.$}
      \label{fig:multipatchigamap}
 \end{figure}
We denote by $F_{ij} = \partial \Omega_i \cap \partial \Omega_j, i \neq j,$ the 
interior facets of two patches see Fig.~\ref{fig:multipatchiga}. We assume the 
$F_{ij} \subset \partial \Omega_i$ for the interior facets. 
Let 
$F_i = \partial \Omega_i \cap \partial \Omega$ 
denote an edge of $\partial \Omega_i.$
$\mathcal{F} := \mathcal{F}_I \cup \mathcal{F}_D.$
  \begin{figure}[th!]
  \centering
      \includegraphics[width=0.8\textwidth]{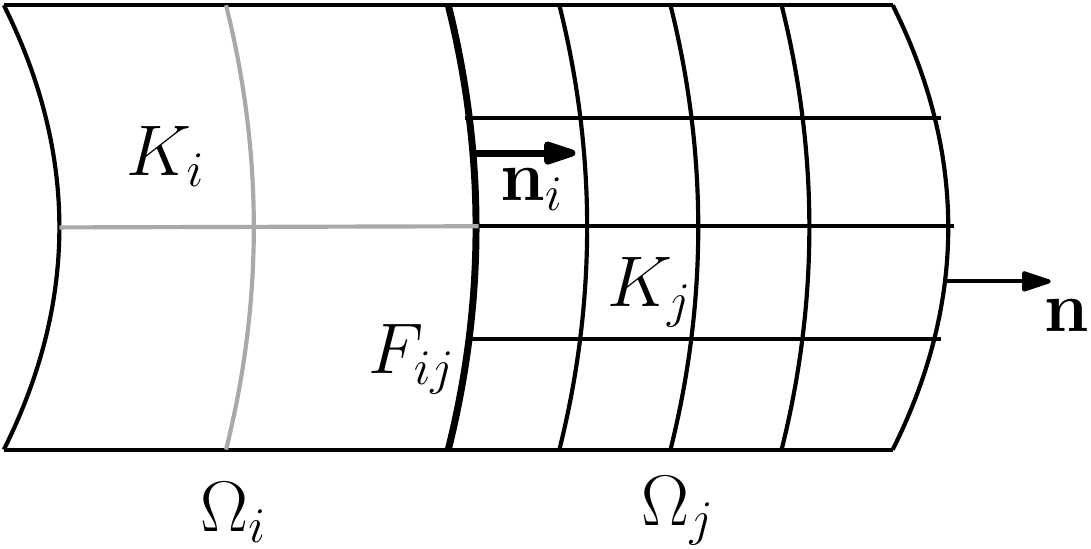}
      \caption{Illustration of the underlying non-matching mesh of the 
multi-patch isogeometric}
      \label{fig:multipatchiga}
 \end{figure}
We assume that for each patch $\Omega_i, i=1,\ldots,N,$ the underlying mesh 
$\mathcal{K}_{h,i}$ is quasi-uniform i.e.  
    $h_K  \leq h_i \leq C_u h_K,$ for all $K \in \mathcal{K}_{h,i}, \quad  i = 
1,\ldots,N,$
where $C_u \geq 1$ and $h_i = \max\{h_K, K \in \mathcal{K}_{h,i}\}$ is the mesh 
size of $\Omega_i$ and $h_K$ is the diameter of of the mesh element $K$ see e.g. 
\cite{Moore:2017a}. 
\section{Discontinuous Galerkin IGA Scheme formulation}
\label{sec:dGIGASchemeFormulation}
We recall some function spaces required for the derivation of interior penalty 
Galerkin schemes.
We assign to each patch $\Omega_i$ a a real number $s_i$ and collect them in the 
vector $\mathbf{s} = \{s_1,\ldots,s_N\}.$ Let us now define the broken Sobolev 
space
\begin{equation}
 \label{eqn:brokensobolevspace}
 H^{\mathbf{s}}(\Omega,\mathcal{T}_h) :=\{v \in L_{2}(\Omega): v|_{\Omega_i} \in 
H^{s_i}(\Omega_i), \; i = 1,\ldots,N\},
\end{equation}
and the corresponding broken Sobolev norm and semi-norm
  \begin{align}
  \label{eqn:brokensobolevnorm}
    \|v \|_{H^{\mathbf{s}}(\Omega,\mathcal{T}_h)} := \left( \sum_{i=1}^N \|v \|^2_{H^{s_i}(\Omega_i)}\right)^{1/2} \quad \text{and} \quad 
     |v |_{H^{\mathbf{s}}(\Omega,\mathcal{T}_h)} := \left( \sum_{i=1}^N   |v |^2_{H^{s_i}(\Omega_i)}\right)^{1/2},
  \end{align} 
  respectively.

We denote the restrictions of the function $v$ on patches $\Omega_i$ and 
$\Omega_j,$ by 
$v_i$ and $v_j$ respectively. For the interior facets $F_{ij} \subset \partial 
\Omega_i,$ let $\mathbf{n}_i$ be the outward unit normal vector with respect to 
$\Omega_i,$ which coincides with the outward unit normal $\normal$ on $\partial 
\Omega,$ see Fig.~\ref{fig:multipatchiga}. We define the jump and average across 
the interior facets $F_{ij}$ of 
a smooth function $v \in H^1(\Omega, \mathcal{T}_h)$ by
\begin{align}
\label{eqn:jumpandaverage}
 \jump{  v} :=  v_i -   v_j , \quad and \quad  
 \{  v \} := \frac{1}{2}\left( v_i  +  v_j \right), 
F_{ij} \in \mathcal{F}_I,
 \end{align}
whereas the jump and average functions on the facets $F_i$ are given by
 $\jump{v} :=   v_i,$ and $\{ v \} :=  v_i. $

Now, we present the dGIGA variational scheme as follows:
find $u \in V = H^{1+\mathbf{s}}(\Omega, \mathcal{T}_h)$ with $\mathbf{s} \in ( 
1/2,1],$ such that 
\begin{equation}
 \label{eqn:VariationalIdentity}
a_h(u,v) = \ell_h(v), \quad \forall v \in V,
\end{equation}
where the dG bilinear and linear forms considered throughout this paper
are defined by the relationships where the bilinear form is given as
\begin{equation}
\label{eqn:dgbilinearform}
 a_h(u,v)= \sum_{i=1}^{N} \big( a_i(u,v)+s_{i}(u,v)+p_{i}(u,v) \big),
\end{equation}
with
\begin{align*}
 a_i(u,v) &:= \int_{\Omega_i} 
		  \alpha_i \nabla_{\Omega} u \cdot \nabla_{\Omega} v + uv \,\, 
dx, \\
 s_{i}(u,v) &:=  \sum_{F_{ij} \subset \partial \Omega_i}
		  \int_{F_{ij}} 
		     \frac{\alpha_{ij}}{2}\bigg( \{\nabla_{\Omega} u\}\jump{v}
		     +  \{\nabla_{\Omega} v\}\jump{u}   \bigg)\,ds, \\
 p_{i}(u,v) &:=  \sum_{F_{ij} \subset \partial \Omega_i} \int_{F_{ij}}  
		      \frac{\delta \alpha_{ij}}{2h_{ij}} \jump{u}\jump{v} \,ds, 
\end{align*}
where $\delta$ is a non-zero positive real number. We have used a harmonic mean 
for the edges on the interface i.e. $h_{ij}=2h_ih_j/(h_i + h_j)$ with $h_{ij} 
\leq 2h_i$ and $h_{ij} \leq 2h_j$ and similarly for the diffusion coefficient 
i.e. 
$\alpha_{ij}=2\alpha_i \alpha_j/(\alpha_i + \alpha_j)$ with $\alpha_{ij} \leq 
2\alpha_i$ and $\alpha_{ij} \leq 2\alpha_j.$ The linear form is given by
\begin{align}
\label{eqn:dglinearform}
\ell_h(v)  &= \sum_{i=1}^N \bigg( \int_{\Omega_i} f v\, dx + \int_{F_i \in 
\mathcal{F}_N} g_N v \,ds\bigg),
\end{align}
where $\mathcal{F}_N$ is the collection of all edges on the Neumann Boundary 
parts.
\begin{remark}
 \label{rem:penaltyparameter}
  The choice of the penalty parameter $\delta$ depends on B-spline or NURBS 
degree $p$ and the dimension of the computational domain $\Omega \subset 
\mathbb{R}^d,$ for example in FEM $\delta=2(p+1)(p+d)/d,$ see e.g. 
\cite{Shahbazi:2005a}.
\end{remark}

\section{Analysis of the dGIGA Scheme}
\label{sec:AnalysisOfDiscretizationError}
For each subdomain $\Omega_i, i=1,\ldots,N,$ 
we will consider the discrete space $ \mathbb{V}_{h,i}, i=1,\ldots,N$ 
where $\mathbb{V}_{h}$ is given by \eqref{eqn:discretefunctionspace}. 
We define the discrete space corresponding to the domain $\Omega$ as 
\begin{equation*}
 V_h := \{v \in    \mathbb{V}_{h,i}, i=1,\ldots,N \},
\end{equation*}
which allows discontinuities across the patch interface.
The discrete dGIGA scheme then reads as: find $u_h \in V_h$ such that 
\begin{align}
\label{eqn:discretedgvariationalform}
 a_h(u_h,v_h) = \ell(v_h), \quad \forall v_h \in V_h.
\end{align}

The existence and uniqueness of the bilinear form $a_h(\cdot,\cdot)$ follow the 
popular Lax-Milgram theorem by showing the coercivity and boundedness. Next, we 
show that the bilinear form $a_h(\cdot,\cdot)$ is $V_{h}-$coercive with respect 
to the dG-norm 
\begin{align}
\label{eqn:dgdiscretenorm}
\|v \|_h^2 := \sum_{i=1}^{N} 
         \bigg( \alpha_i \|\nabla_{\Omega} v_i \|_{L_{2}(\Omega_i)}^2  
         + \| v \|_{L_{2}(\Omega_i)}^2
        & + \sum_{ F_{ij} \subset \partial \Omega_i}
	  \frac{\delta \alpha_{ij}}{2h_{ij}} \|\jump{v}\|_{L_2(F_{ij})}^2 \bigg).
\end{align}
\begin{remark}
 \label{rem:dgdiscretenorm}
 The discrete norm $\| \cdot \|_h$ from \eqref{eqn:dgdiscretenorm} is a norm on 
$V_h.$
 Indeed, if $\|v \|_h = 0$ for some function $v \in V_{h}$, then 
$\nabla_{\Omega} v = 0$ in each subdomain $\Omega_i$. This means that the 
function $v$ is a constant on each patch $\Omega_i$,  $i = 1,\ldots,N$.  
Furthermore, $\|v \|_h = 0$ yields that $\jump{v} = 0$ across the internal 
facets $F_{ij} \subset \partial \Omega_i$ are zero, i.e., $v$ is constant in 
$\overline{\Omega}$.
Finally, $v_i = 0$ on the boundary $\partial \Omega$ implies that this constant 
must be zero. Thus, $v=0$ in $\overline{\Omega}$. The other norm axioms are 
obviously fulfilled.
\end{remark}

To analyze the multi-patch interior penalty Galerkin scheme, the following 
discrete inverse and trace inequalities are required.  
\begin{lemma}
\label{lem:pinterfaceinverseinequalities}
Let $v \in V_h,$ then the following inverse inequalities hold;
\begin{equation}
\label{eqn:pinterfaceinverseinequalities}
  \| \nabla v \|_{L_2(\Omega_i)} \leq C_{inv,1,u} h^{-1}_i \| v 
\|_{L_{2}(\Omega_i)},
\end{equation}
and 
\begin{equation}
 \label{eqn:pinterfaceinverseinequalities2}
  \|  v \|_{L_2(\partial \Omega_i)} \leq C_{inv,0,u} h^{-1/2}_i \| v 
\|_{L_2(\Omega_i)},
\end{equation}
where $C_{inv,1,u}$ and $C_{inv,0,u}$ are positive constants, which are 
independent of $h_i$ and $\Omega_i$.
\end{lemma}
We conclude with the  continuous trace inequality,
 \begin{lemma}
 \label{lem:patchtraceinequality}
 Let $\Omega_i = \Phi_i (\widehat{\Omega})$ for $i = 1,\dots,N$.
 Then the patch-wise scaled trace inequality 
    \begin{equation}
    \label{eqn:patchtraceinequality}
      \|v \|_{L_2(\partial \Omega_i)} \leq C_{t,u} h_i^{-1/2} 
             \bigg( \|v \|_{L_2(\Omega_i)} + 
h_i^{1/2+\epsilon}|v|_{H^{1/2+\epsilon}(\Omega_i)} \bigg),  
    \end{equation}
 holds for all $v \in H^{1/2+\epsilon}(\Omega_i), \epsilon \in (0,1/2],$ where 
$h_i$ denotes the maximum mesh size in the physical domain, and $C_{t,u}$ is a 
positive constant that only depends on the 
shape regularity of the mapping $\Phi_i.$ 
\end{lemma}
The proofs of Lemma~\ref{lem:pinterfaceinverseinequalities} and 
Lemma~\ref{lem:patchtraceinequality} follows the standard procedure see e.g., 
\cite{EvansHughes:2013a} and \cite{Moore:2017a}. 
from the Finite Element.
\begin{lemma}
 \label{lem:consistencytermestimate}
For an arbitrary positive $\varepsilon$ and for $ F_{ij} \subset \partial 
\Omega_i$ the estimates
\begin{align}   
   \bigg| \int_{F_{ij} }& \frac{\alpha_{ij}}{2}\normal_i \cdot \nabla_{\Omega} 
v_{h,i} \jump{v_h} \,ds  \bigg|
    \leq\bigg( \alpha_i\varepsilon  \|\nabla_{\Omega} v_{h,i}  
\|_{L_2(\Omega_i)}^2   + \frac{C_{inv,0,u}^2 \alpha_{ij}}{2\varepsilon h_{ij} }
   \|\jump{v_h}\|_{L_2(F_{ij})}^2 \bigg), \label{eqn:consistencytermestimate}    
 \end{align}
holds for all $v_{h,i}, v_{h,j} \in V_{h}$, a positive constant $C_{inv,0,u}$ 
and $\alpha_i >0.$
\end{lemma}
\begin{proof}
Following the Cauchy Schwarz inequality, since $F_{ij} \subset \partial 
\Omega_i,$ by using the trace inequality \eqref{eqn:pinterfaceinverseinequalities2},  we 
have 
\begin{align*}
 \bigg| \int_{F_{ij} } \frac{\alpha_{ij}}{2}\normal_i \cdot \nabla_{\Omega} 
v_{h,i}\jump{v_h}  \,ds\bigg|  
  & \leq \frac{\alpha_{ij}}{2} ||  \nabla_{\Omega} v_{h,i}||_{L_2(F_{ij})} 
  ||\jump{v_h} ||_{L_2(F_{ij})} \\
  & \leq C_{inv,0,u}\frac{\alpha_{ij}}{2h_{i}^{1/2}} ||  \nabla_{\Omega} 
v_{h,i}||_{L_2(\Omega_i)} 
  ||\jump{v_h} ||_{L_2(F_{ij})}\\
  & \leq \frac{C_{inv,0,u}\alpha_{ij}}{h_{ij}^{1/2}} ||  \nabla_{\Omega} 
v_{h,i}||_{L_2(\Omega_i)} 
  ||\jump{v_h} ||_{L_2(F_{ij})}\\
  &\leq \bigg( \frac{\alpha_{ij} \varepsilon }{2} \|\nabla_{\Omega} v_{h,i}  
\|_{L_2(\Omega_i)}^2   
  + \frac{C_{inv,0,u}^2 \alpha_{ij}}{2\varepsilon h_{ij} }
   \|\jump{v_h}\|_{L_2(F_{ij})}^2 \bigg),
\end{align*}
where we use $h_{ij} \leq 2h_i$ and $\alpha_{ij} \leq 2\alpha_i,$ together with
the inequality $ab \leq \varepsilon a^2/ 2 + b^2/(2\varepsilon), \forall a,b \in 
\mathbb{R}$ with $\varepsilon > 0$ to complete the proof. \qed
\end{proof}
Using the above result, we proceed to show the coercivity of the bilinear form 
$a_h(\cdot,\cdot).$
{\allowdisplaybreaks
\begin{lemma}[Coercivity]
\label{lem:dgbilinearcoercivity}
Let $a_h(\cdot,\cdot): V_h \times V_h \rightarrow \mathbb{R} $ be the bilinear 
form \eqref{eqn:dgbilinearform}. There exists  $\delta_0 > 0$ and $\mu_c > 0 $ 
such that $\delta > \delta_0$ and 
the discrete bilinear form $a_h(\cdot,\cdot)$ is $V_h-$coercive with respect to 
the norm $\|\cdot\|_h$, i.e.  
  \begin{equation}
  \label{eqn:dgbilinearcoercivity}
     a_h(v_h ,v_h) \geq \mu_c \|v_h \|^2_h, \quad \forall v_h \in V_{h},
  \end{equation}
where $\mu_c$ is independent of $\alpha_i, h_i$ and $N.$
\end{lemma}
\begin{proof}
By using Cauchy-Schwarz's inequality, we proceed as follows
  \begin{align*}
   a_h(v_h,v_h) & = \|v_h \|_h^2  
	  - 2\sum_{i=1}^{N} \bigg( \sum_{F_{ij} \subset \partial \Omega_i} 
		  \int_{F_{ij} } \frac{\alpha_{ij}}{2} \{ \nabla_{\Omega}v_{h} \} \jump{v_h} \,ds  \bigg).  
  \end{align*}
	 Using Cauchy Schwarz's inequality and Lemma~\ref{lem:consistencytermestimate}, we have 
	 \begin{align*}
       a_h(v_h,v_h) \geq &  \left(1 - \frac{\varepsilon}{2} \right)\sum_{i=1}^{N}  \alpha_i \| \nabla_{\Omega} v_{h,i} \|_{L_2(\Omega_i)}^2 
         +\left(\delta - \frac{ C^2_{inv,0,u}}{ \varepsilon} \right) 
\sum_{i=1}^{N} \sum_{F_{ij}  \subset \partial \Omega_i} 
\frac{\alpha_{ij}}{2h_{ij}} \| \jump{v_h} \|_{L_2(F_{ij})}^2.
   \end{align*}
   For example, for $\mu_c = 1/2,$ we choose $\varepsilon = 1$ and $\delta \geq 
C^2_{inv,0,u}.$
\qed 
\end{proof}
Next, we prove the uniform boundedness for the bilinear form $a_h(\cdot,\cdot)$ 
on $V_{h,*} \times V_{h},$ where  
 $V_{h,*}= V_0 \cap H^{1+\mathbf{s}}(\Omega,\mathcal{T}_h) + V_{h}$ with 
$\mathbf{s} > 1/2$ is equipped with the norm
 \begin{equation}
  \label{eqn:dgdiscretenorm*}
   \|v \|_{h,*} = \left( \|v \|_h ^2 + \sum_{i=1}^N \alpha _i h_i 
\|\nabla_{\Omega} v_i\|_{L_2(\partial \Omega_i)}^2 \right)^{1/2}.
 \end{equation}

To prove \textit{a priori} estimates, we need to show the uniform boundedness of 
the bilinear form. We need the following two auxiliary lemmata to proof for the 
boundedness of the bilinear form.  
 {\allowdisplaybreaks
\begin{lemma}
 \label{lem:consistencytermestimateboundedness}
 For a positive parameter $\delta$ and for $F_{ij} \subset \partial \Omega_i, 
i=1,\ldots,N$ and diffusion coefficients $\alpha_i$ and $\alpha_{ij},$ the 
estimates 
 \begin{align}
	\bigg|\int_{F_{ij} } \frac{\alpha_{ij}}{2} \normal_i \cdot 
\nabla_{\Omega} u_{i} \jump{v_h} \,ds\bigg|
	& \leq 
    \bigg( \frac{2\alpha_i h_i }{\delta} \|\nabla_{\Omega} u_i \|_{L_2(\partial 
\Omega_i)}^2 \bigg)^{1/2}  \bigg( \frac{\alpha_{ij}\delta}{2 h_{ij}}  
\|\jump{v_h}\|_{L_2(F_{ij})}^2 
\bigg)^{1/2},\label{eqn:consistencytermestimateboundedness} \\
    	\bigg| \int_{F_{ij}} \frac{\alpha_{ij}}{2} \normal_i \cdot 
\nabla_{\Omega} v_{h,i} \jump{u} \,ds \bigg|
		 & \leq 
    \bigg( \frac{2 C^2_{inv,0,u}\alpha_i }{ \delta}  \|\nabla_{\Omega} v_{h,i} 
\|_{L_2(\Omega_i)}^2 \bigg)^{1/2} \bigg(   \frac{\alpha_{ij}\delta}{ 2h_{ij}}
   \|\jump{u} \|_{L_2(F_{ij})}^2 
\bigg)^{1/2},\label{eqn:consistencytermestimateboundedness-2}
 \end{align}
 hold for all $u \in V_{h,*}$ and for all $v_h \in V_{h}.$
\end{lemma}
\begin{proof}
 For $F_{ij} \subset \partial \Omega_i,$ using Cauchy-Schwarz's inequality, we 
obtain
 \begin{align*}
   \bigg|\int_{F_{ij} } \frac{\alpha_{ij}}{2} \normal_i \cdot \nabla_{\Omega} 
u_{i} \jump{v_h} \,ds\bigg| 
    \leq 
         \bigg(\frac{\alpha_{ij} h_{ij}}{2\delta} 
		      \| \nabla_{\Omega} u_i  \|^2_{L_2(\partial \Omega_i)}  
\bigg)^{\frac{1}{2}} 
		      \bigg(\frac{\alpha_{ij} \delta}{2h_{ij}}
		      \| \jump{v_h} \|^2_{L_2(F_{ij})}\bigg)^{\frac{1}{2}}.
\end{align*}
We conclude the proof since $h_{ij} \leq 2 h_i$ and $\alpha_{ij} \leq 2 
\alpha_i$ For the second inequality, we apply the Cauchy Schwarz inequality to 
obtain
 \begin{align*}
   \bigg| \int_{F_{ij}} \frac{\alpha_{ij}}{2} \normal_i \cdot \nabla_{\Omega} 
v_{h,i} \jump{u} \,ds \bigg| 
    \leq 
        \bigg(\frac{\alpha_{ij} h_{ij}}{2\delta} 
		      \| \nabla_{\Omega} v_{h,i}  \|^2_{L_2(\partial \Omega_i)} 
\bigg)^{\frac{1}{2}}  
		      \bigg(\frac{\alpha_{ij} \delta}{2h_{ij}}
		      \| \jump{u} \|^2_{L_2(F_{ij})}\bigg)^{\frac{1}{2}}.
\end{align*}
Since $h_{ij} \leq 2 h_i$ and $\alpha_{ij} \leq 2 \alpha_i,$ by applying the 
inequality \eqref{eqn:pinterfaceinverseinequalities2} for $v_{h,i} \in V_h,$ we 
complete the proof.
\qed 
\end{proof} 
}

Next, we proceed with the boundedness of the bilinear form $a_h(\cdot,\cdot)$ as 
follows ;
{\allowdisplaybreaks
\begin{lemma}[Boundedness]
\label{lem:dgbilinearboundedness}
The discrete bilinear form $a_h(\cdot,\cdot)$ is uniformly bounded on
 $V_{h,*} \times V_{h},$ i.e. there exists a 
 mesh-independent  positive constant $\mu_b$ 
 such that 
  \begin{equation}
  \label{eqn:dgbilinearboundedness}
    | a_h(u,v_h) |\leq \mu_b \|u \|_{h,*} \|v_h \|_h, \forall u \in V_{h,*} , 
\forall v_h \in V_{h}.
  \end{equation}
\end{lemma}
\begin{proof}
The first term of the bilinear form \eqref{eqn:dgbilinearform} is estimated 
using Cauchy-Schwarz's inequality as follows
  \begin{align}
  \label{eqn:stabilityconsistencyterm-0}
   \left|\sum_{i=1}^N a_i(u,v_h) \right| & \leq \left( \sum_{i=1}^N \alpha_i \| 
\nabla_{\Omega} u \|^2_{L_2(\Omega_i)}\right)^{\frac{1}{2}}\left( \sum_{i=1}^N 
\alpha_i  \| \nabla_{\Omega} v_h\|^2_{L_2(\Omega_i)}  \right)^{\frac{1}{2}}. 
  \end{align}
   Using Cauchy Schwarz's inequality together with 
Lemma~\ref{lem:consistencytermestimateboundedness}, the second term yields
   \begin{align}
   \label{eqn:stabilityconsistencyterm}
      \bigg| \sum_{i=1}^N s_i(u,v_h) \bigg|  
      &\leq \bigg( \sum_{i=1}^N \sum_{ F_{ij}  \subset \partial  \Omega_i}
      \bigg[\frac{2\alpha_i h_i}{\delta}\|\nabla_\Omega u \|^2_{L_2(\partial 
\Omega_i)} 
      +  \frac{ \delta \alpha_{ij}}{2 h_{ij}}\|\jump{u}\|_{L_2(F_{ij})}\bigg] 
\bigg)^{\frac{1}{2}} \notag \\ 
     & \times \bigg( \sum_{i=1}^N \sum_{ F_{ij}  \subset \partial  \Omega_i}
      \frac{ 2C_{inv,0,u}^2 \alpha_i}{\delta}\|\nabla_\Omega v_h 
\|^2_{L_2(\Omega_i)} 
      +  \frac{\delta \alpha_{ij} }{2 h_{ij}}\|\jump{v_h}\|_{L_2(F_{ij})} 
\bigg)^{\frac{1}{2}}.
     \end{align}
Also, the last term of the bilinear form is estimated by applying 
Cauchy-Schwarz's inequality to obtain 
   \begin{align}
   \label{eqn:penaltytermestimate}
    \bigg|\sum_{i=1}^N p_i(u,v_h) \bigg| 
    & \leq \bigg(\sum_{i=1}^N \sum_{ F_{ij}  \subset \partial  \Omega_i} 
\frac{\delta \alpha_{ij} }{ 2 h_{ij} }\|\jump{u} \|^2_{L_2(F_{ij})} 
\bigg)^{\frac{1}{2}}
    \bigg(\sum_{i=1}^N \sum_{ F_{ij}  \subset \partial  \Omega_i} \frac{\delta 
\alpha_{ij} }{ 2 h_{ij} }\|\jump{v_h} \|^2_{L_2(F_{ij})} \bigg)^{\frac{1}{2}}.
   \end{align}
Combining all the terms \eqref{eqn:stabilityconsistencyterm-0} -- 
\eqref{eqn:penaltytermestimate}, we conclude the proof with the positive 
constant $\mu_b  = 2 \sqrt{\max\{ 1, (1 + C_{inv,0,1}^2/\delta) \} }.$ 
\qed 
\end{proof}

We note that the discrete norms $\|\cdot\|_h$ and $\|\cdot\|_{h,*}$ are 
uniformly equivalent on the dicrete space $V_h.$ 
In the next lemma, we present this equivalence of the discrete norms 
since the convergence analysis is considered in the discrete norm $\|\cdot 
\|_h.$

\begin{lemma}
 \label{lem:normequivalence}
 The norms $\|\cdot\|_h$ and $\|\cdot\|_{h,*}$ are uniformly equivalent on the 
discrete space $V_h$ such that 
 \begin{equation}
  \label{eqn:normequivalence}
   C_{e}\|v_h \|_{h,*} \leq \|v_h \|_h \leq \|v_h \|_{h,*}, \quad \forall v_h \in V_h,
 \end{equation}
 where $C_e$ is mesh independent.
\end{lemma}
\begin{proof}
The proof of the upper bound follows immediately. However, the proof of the lower bound 
follows by using the definition of the norm 
\eqref{eqn:dgdiscretenorm*}
together with the trace inequality \eqref{eqn:pinterfaceinverseinequalities2}
with $C_e = \left(1 + C_{inv,0,1}^2/\delta \right)^{-1}.$
\qed
\end{proof}

A consequence of \lemref{lem:normequivalence} yields the boundedness of the 
bilinear form $a_h(\cdot,\cdot)$ with a positive constant $\tilde{\mu}_b$ as 
follows
\begin{align}
  \label{eqn:dgboundedness-2}
  | a_h(u_h,v_h)| \leq \tilde{\mu}_b \|u_h\|_h \|v_h\|_h, \quad \forall u_h, v_h 
\in V_h,
 \end{align}
 where $\tilde{\mu}_b = \mu_b\left(1 +C_{inv,0,1}^2/\delta \right)^{-1/2}.$
\section{Error Analysis of dGIGA Discretization}
\label{sec:erroranalysisOfdgigadiscretization}
Finally, we present the approximation estimates required to obtain \textit{a 
priori} 
error estimates. For patch $\Omega_i,i=1,\ldots,N,$  
let $\Pi_{h,i} : L_2(\Omega_i) \rightarrow \mathbb{V}_{h,i}$ 
denote a quasi-interpolant that yields optimal approximation results. 
Of course, such an interpolant is known to exist and has been well studied and 
presented in
\cite{BazilevsBeiraoCottrellHughesSangalli:2006a,DaVeigaBuffaSangalliVazquez:2014a} as follows
}
\begin{lemma}
\label{lem:localerrorestimate} 
Let $l$ and $s$ be integers with $0 \leq l \leq s \leq p+1$ 
 and $K \in \mathcal{K}_{h,i}$. Then there exist an interpolant $\Pi_{h,i} v \in 
\mathbb{V}_{h,i}$ 
 for all $ v \in H^{s}(\Omega_i)$ and a constant $C_s > 0$ such that the 
following inequality holds 
 \begin{align}
 \label{eqn:localerrorestimate}
       \sum_{K \in \mathcal{K}_{h,i}}|v - \Pi_{h,i}v|^{2}_{H^{l}(K)} 
      & \leq C_s  h_i^{2(s-l)} \|v \|^2_{H^{s}(\Omega_i)}, 
 \end{align} 
where $h_i$ is the mesh size in the physical domain, and $p$ denotes the 
underlying polynomial degree 
of the B-spline or NURBS.
\end{lemma}

For patch $\Omega_i,i=1,\ldots,N,$ the local estimate 
\eqref{eqn:localerrorestimate} yields a global estimate if the multiplicity of 
the inner knots is not larger than $p + 1 - l$ and $\Pi_{h,i} v \in 
\mathbb{V}_{h,i} \cap H^{l}(\Omega_i).$ 
\begin{proposition}
\label{prop:globalerrorestimate}
Let $v \in H^{s}(\Omega_i)$ be a function defined in the physical domain 
$\Omega_i.$ 
Given an integer $l$ such that $0 \leq l \leq p+1, l \leq s,$ and $p \leq s+1,$ 
where $s$ is the smoothness of the considered B-Spline basis. Then there exists 
a projection operator $\Pi_{h,i} : L_2(\Omega_i) \rightarrow \mathbb{V}_{h,i}$ 
such that the approximation error estimates
 \begin{equation}
 \label{eqn:globalerrorestimate}
   |v - \Pi_{h,i}v|_{H^{l}(\Omega_i)}  
   \leq C_s h_i^{(\beta-l)} \|v \|_{H^{s_i}(\Omega_i)}, 
 \end{equation} 
where $\beta = \min\{p+1,l\},$ $h_i$ denotes the maximum mesh-size parameter in 
the physical domain and the generic constant $C_s$ only depends on $l,s$ and 
$p$, the shape regularity of the physical domain $\Omega_i$ described by the mapping 
$\Phi$ and, in particular, $\nabla_{\Omega} \mathrm{\Phi}.$
\end{proposition}
\begin{proof}
See \cite[Proposition~3.2]{DedeQuarteroni:2015a}.
\end{proof}
For the error analysis, we assume that the patches have the same regularity 
such that $\mathbf{s}=\{s_1,s_2,\ldots,s_N\}=s$ and $H^{1+s}(\Omega_i), i=1,\ldots,N.$
\begin{lemma}
\label{lem:interfaceapproximationerror}
Let $v \in  V_0 \cap  H^{1+s}(\Omega_i)$ with $s >1/2$ and $p \ge 1.$ 
By assuming quasi-uniform meshes, then there exists a projection 
$\Pi_{h}v \in V_{h}$ and generic positive constants $C_0$ and $C_1$ 
such that the following error estimates hold
  \begin{align}
      \sum_{i=1}^{N} \sum_{F_{ij} \subset \partial \Omega_i} \frac{\delta 
\alpha_{ij}}{2 h_{ij}} \|\jump{ v - \Pi_{h}v}  \|_{L_{2}(F_{ij})}^2 
    &  \leq C_0 \sum_{i=1}^{N} \left(  h_i^{2r} + \sum_{F_{ij} \subset \partial 
\Omega_i}   \frac{h_j}{h_i}h_j^{2r}\right) \alpha_i\| v\|^2_{H^{1+r}(\Omega_i)}, 
 \label{eqn:interfaceapproximationerror1} \\
  \sum_{i=1}^N \alpha_i h_i  \|  \nabla_{\Omega} (v - \Pi_{h}v)  
\|^2_{L_{2}(\partial \Omega_i)}  
   &  \leq C_1  \sum_{i=1}^{N} \alpha_i h_i^{2r-1}  \| 
v\|^2_{H^{1+r}(\Omega_i)}, \label{eqn:interfaceapproximationerror2} 
  \end{align}
where $r =\min\{s,p\},$ the constant $C_0$ and $C_1$ $\Omega_i,$ 
are independent of $h_i$ and $h_j$ and $F_{ij} \subset \partial \Omega_i$ 
are the interior facets. 
\end{lemma}
\begin{proof}
  By using \eqref{eqn:jumpandaverage} with \lemref{lem:patchtraceinequality} and 
 \propref{prop:globalerrorestimate}, we estimate the first term as follows
\begin{align}
\label{eqn:interfacejumpbound2-2}
  \frac{\delta \alpha_{ij}}{2 h_{ij}} \| v-\Pi_{h,i}v  \|^2_{L_{2}(F_{ij})} 
  &\leq 
 \frac{\delta \alpha_{ij}}{2 h_{ij}}
 C_{t,u}^2\bigg( h_i^{-1}  \| v-\Pi_{h,i}v  \|_{L_{2}(\Omega_i)}^2 + 
h_i^{2\epsilon +1}  | v-\Pi_{h,i}v  |_{H^{1/2+\epsilon}(\Omega_i)}^2 \bigg) 
\notag \\
 & \leq 2 \delta C_{t,u}^2C_s \frac{\alpha_i}{2 h_{ij}} \bigg( h_i^{-1} 
h_i^{2(1+r)-1/2+\epsilon} + h_i^{2\epsilon +1} h_i^{2r} \bigg)\| v 
\|^2_{H^{1+r}(\Omega_i)} \notag \\
  & \leq 4 \delta C_{t,u}^2 C_s  \frac{\alpha_i}{2 h_{ij}}h_i^{2r-1}\| v 
\|^2_{H^{1+r}(\Omega_i)}. 
\end{align}
Similarly, the second term  
\begin{align}
\label{eqn:interfacejumpbound2-3}
 \frac{\delta \alpha_{ij}}{2 h_{ij}} \| v-\Pi_{h,j}v  \|^2_{L_{2}(F_{ij})} 
 & \leq 4 \delta C_{t,u}^2 C_s  \frac{\alpha_i}{2 h_{ij}}h_j^{2r-1}\| v 
\|^2_{H^{1+r}(\Omega_j)}. 
\end{align}
Now, we complete the proof by summing with respect to the interior facets 
$F_{ij} \subset \partial \Omega_i$ and $i = 1,2,\ldots,N,$ to obtain 
\begin{align}
\label{eqn:interfacejumpbound2}
  \sum_{i=1}^{N}  \sum_{F_{ij} \subset \partial \Omega_i}
  \frac{\delta \alpha_{ij}}{2 h_{ij}}\| \jump{ v - \Pi_{h}v } 
\|_{L_{2}(F_{ij})}^2 
 & \leq C_0 \sum_{i=1}^{N} \alpha_i \left(  h_i^{2r} +   \sum_{F_{ij} \subset 
\partial \Omega_i} \frac{h_j}{h_i}h_j^{2r}\right)\| v\|^2_{H^{1+r}(\Omega_i)}.
\end{align}
where $C_0 = \delta C_{t,u}^2 C_s.$
The proof of \eqref{eqn:interfaceapproximationerror2} follows by using 
Lemma~\ref{lem:patchtraceinequality} and the approximation estimate of 
Proposition~\ref{prop:globalerrorestimate} as follows
\begin{align}
  \sum_{i=1}^N \alpha_i h_i  \|  \nabla_{\Omega} (v - \Pi_{h,i}v)  
\|^2_{L_{2}(\partial \Omega_i)} 
 &\leq  \sum_{i=1}^N\alpha_i C_{t,u}^2 C_s \left(  h_{i}^{2r} + h_{i}^{2\epsilon 
+1}  h_{i}^{2(1+r-3/2-\epsilon)} \right)\| v \|^2_{H^{1+r}(\Omega_i)} \notag \\
 & \leq 2C_s C_{t,u}^2\sum_{i=1}^N \alpha_i h_{i}^{2r} \| 
v\|_{H^{1+r}(\Omega_i)}^2,
\end{align}
where $C_1 = 2C_s C_{t,u}^2.$
\qed 
\end{proof}

To derive the \textit{a priori} error estimate, we show that the interpolant yields the 
optimal approximation estimate in the discrete norms. 
\begin{lemma}
\label{lem:dginterpolationerrorestimate}
Let $v \in  V_0 \cap  H^{1+s}(\Omega_i)$ 
with $s > 1/2$ and $p\ge 1.$
Then there exists a projection $\Pi_{h}v \in V_{h}$ and generic positive 
constants $C_2$ and $C_3$ such that
 \begin{align}
     \|v- \Pi_{h}v \|^2_h 
   & \leq C_2 \sum_{i=1}^{N}  \left(  h_i^{2r} + \sum_{F_{ij}  \subset \partial 
\Omega_i} \frac{h_j}{h_i} h_j^{2r}\right)\alpha_i\| v\|^2_{H^{1+r}(\Omega_i)},  
\label{eqn:dginterpolationerrorestimate} \\
        \|v- \Pi_{h}v \|^2_{h,*} 
   & \leq C_3 \sum_{i=1}^{N}  \left( h_i^{2r} + \sum_{F_{ij} \subset \partial 
\Omega_i} \frac{h_j}{h_i} h_j^{2r} \right)\alpha_i \| v\|^2_{H^{1+r}(\Omega_i)}, 
 \label{eqn:dginterpolationerrorestimate*} 
 \end{align}
 where $F_{ij} \subset \partial \Omega_i$ are the interior facets, of $\Omega_i,$ 
 $r =\min\{s,p\}$ and $C_2$ and $C_3$ only depend on $s$ and $p.$
\end{lemma}
\begin{proof}
 Following from the definition of the discrete norms \eqref{eqn:dgdiscretenorm} 
and \eqref{eqn:dgdiscretenorm*} together with 
\lemref{lem:interfaceapproximationerror}, we complete the proof.
\qed
\end{proof}

Finally, we prove the main result in this section, namely \textit{a priori} 
error estimate for 
surfaces. We will present the results for the discrete norm $\|\cdot \|_h$ and 
the $\|\cdot\|_{L_2(\Omega)}-$norm. 
\begin{theorem}
 \label{thm:discretenormerrorestimate}
  Let $u \in V_0 \cap  H^{1+s}(\Omega_i)$  with $ s > 1/2$ be the exact 
solution of the model \eqref{eqn:variationalformulation} 
  and $u_h \in V_{h}$ with $p\ge 1$ be the discrete solution of the dGIGA scheme 
\eqref{eqn:discretedgvariationalform}. For the penalty parameter $\delta$ chosen 
as in \lemref{lem:dgbilinearcoercivity}, then the discretization error estimate 
  \begin{equation}
   \label{eqn:discretenormerrorestimate}
    \|u - u_h \|^2_h \leq C \sum_{i=1}^{N}  \left( h^{2r}_i 
    + \sum_{F_{ij}  \subset \partial  \Omega_i}  \frac{h_j}{h_i} h_j^{2r}\right) 
\alpha_i\|u\|^2_{H^{1+r}(\Omega_i)},  
  \end{equation}
  holds true, where $r =\min\{s,p\}$ and $p$ denotes the underlying NURBS degree 
of the patch $\Omega_i,$ and $C$ is a positive constant independent of the $h_i$ 
and $h_j.$ 
\end{theorem}
 \begin{proof}
  By using the coercivity result \lemref{lem:dgbilinearcoercivity}, Galerkin 
orthogonality \eqref{eqn:galerkinorthogonality} and the boundedness of the 
discrete bilinear form, \lemref{lem:dgbilinearboundedness}, we obtain
 \begin{align}
 \label{eqn:galerkinorthogonality}
     \mu_c \|\Pi_{h} u - u_h\|_h^2  \leq a_h(\Pi_{h} u - u_h,\Pi_{h} u - u_h)  
              & =  a_h(\Pi_{h} u - u,\Pi_{h} u - u_h) \notag \\
	      & \leq  \mu_b \|\Pi_{h} u - u\|_{h,*} \|\Pi_{h} u - u_h\|_{h}. 
 \end{align}
Thus, we have
\begin{align}
\label{eqn:dgdiscreteinterpolateenergy}
  \|\Pi_h u - u_h\|_h^2 & \leq \left(\mu_b/\mu_c\right)^2 \|\Pi_h u - u\|_{h,*}^2 
\end{align}
Using \lemref{lem:interfaceapproximationerror}, we get 
\begin{align*}
  \|u - u_h\|_h^2 & \leq \|u - \Pi_h u\|_h^2 + \|\Pi_h u - u_h\|_h^2  \\
  & \leq C_3 \sum_{i=1}^{N} \left(  h_i^{2r} +  \sum_{F_{ij}  \subset \partial  
\Omega_i}\frac{h_j}{h_i} h_j^{2r}\right)\alpha_i\|u\|^2_{H^{1+r}(\Omega_i)} \\ 
  & \qquad + (\mu_b/\mu_c)^2C_4 \sum_{i=1}^{N} \left(  h_i^{2r} +  \sum_{F_{ij}  
\subset \partial  \Omega_i}  \frac{h_j}{h_i} 
h_j^{2r}\right)\alpha_i\|u\|^2_{H^{1+r}(\Omega_i)} \\
  & = C \sum_{i=1}^{N} \left(h_i^{2r} + \sum_{F_{ij}  \subset \partial  
\Omega_i} \frac{h_j}{h_i} h_j^{2r}\right)\alpha_i\|u\|^2_{H^{1+r}(\Omega_i)} ,
\end{align*}
where $C = (C_3 + \left(\mu_b/\mu_c\right)^2 C_4).$ 
\lemref{lem:dgbilinearboundedness}.
\qed
\end{proof}

\begin{remark}
 \label{rem:discretenormerrorestimate}
 If, we assume matching meshes i.e. $h_i = h_j,$ then the \textit{a priori} 
error estimate 
 \eqref{eqn:discretenormerrorestimate} yields 
   \begin{equation}
   \label{eqn:matchingdiscretenormerrorestimate}
    \|u - u_h \|^2_h \leq C \sum_{i=1}^{N}\alpha_i h^{2r}_i 
\|u\|^2_{H^{1+r}(\Omega_i)},  
  \end{equation}
 which has been studied and presented in \cite{LangerMoore:2014a}.
\end{remark}

\section{Numerical Results}
\label{sec:NumericalResults}
In this section, we present numerical results for the dGIGA scheme and \textit{a priori} 
error estimate of \thmref{thm:discretenormerrorestimate}. All the numerical 
experiments have been performed in \gismo see 
\cite{JuettlerLangerMantzaflarisMooreZulehner:2014a}. 
We solve the linear system arising from the dGIGA formulation by means a
preconditioned conjugate gradient (PCG) algorithm with a scaled Dirichlet 
preconditioner where we choose vertex evaluation and edge averages as primal variables in the 
so-called dual-primal isogeometric tearing and interconnecting (dG-IETI-DP) solver. The 
solver is known to be robust with respect to diffusion coefficient see e.g. 
\cite{Hofer:2018}. A reduction of the initial residual factor of $10^{-6}$ is used as a stopping 
criterion together with a zero initial guess. In the examples, we present 
non-matching grid of ratio $h_i/h_j = 2^q,$  where $q$ is the mesh refinement. 
The ratio $h_i/h_j$ denotes the relative number of refinement on the neighboring 
patches and $h_i,h_j$ are the maximum mesh sizes of patches $\Omega_i$ and 
$\Omega_j.$ The penalty parameter is chosen to be $\delta = 2(p+2)(p+1),$ where 
$p$ is the NURBS degree.
The convergence rate is computed using the formula $rate = \log_2 \left( e_{i+1}/e_{i} \right),$ where 
  $e_{i+1} = \|u -u_{h,i+1}\|_h$ and $e_{i} = \|u -u_{h,i}\|_h$ to study the discrete solution of 
  the model problem. We consider as computational domains a quarter cylinder and a torus for the 
  open and closed surfaces respectively, see Fig.~\ref{fig:computationaldomains}.
  \begin{figure}[thb!]
  \centering
      \includegraphics[width=0.45\textwidth]{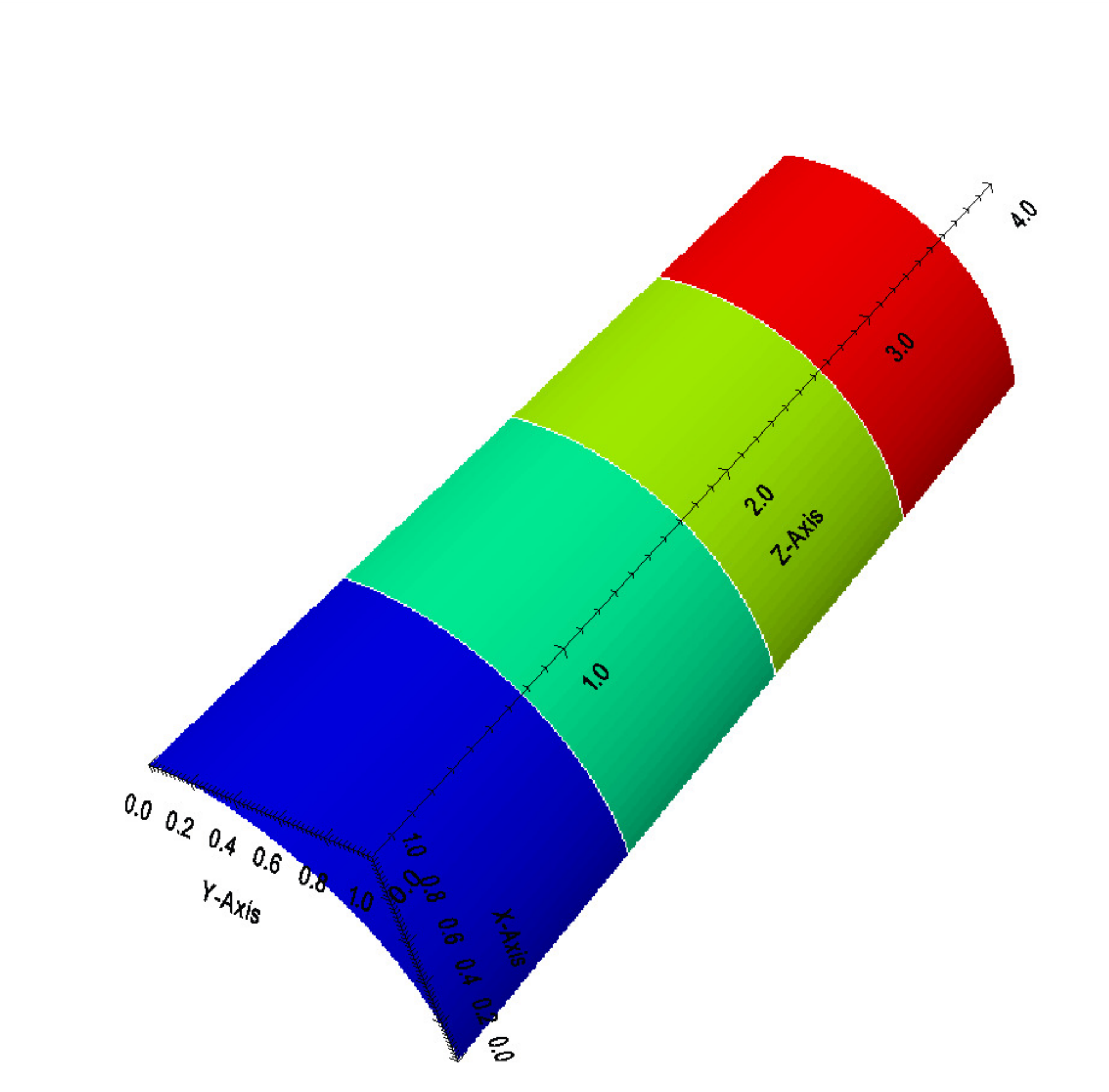}
      \includegraphics[width=0.45\textwidth]{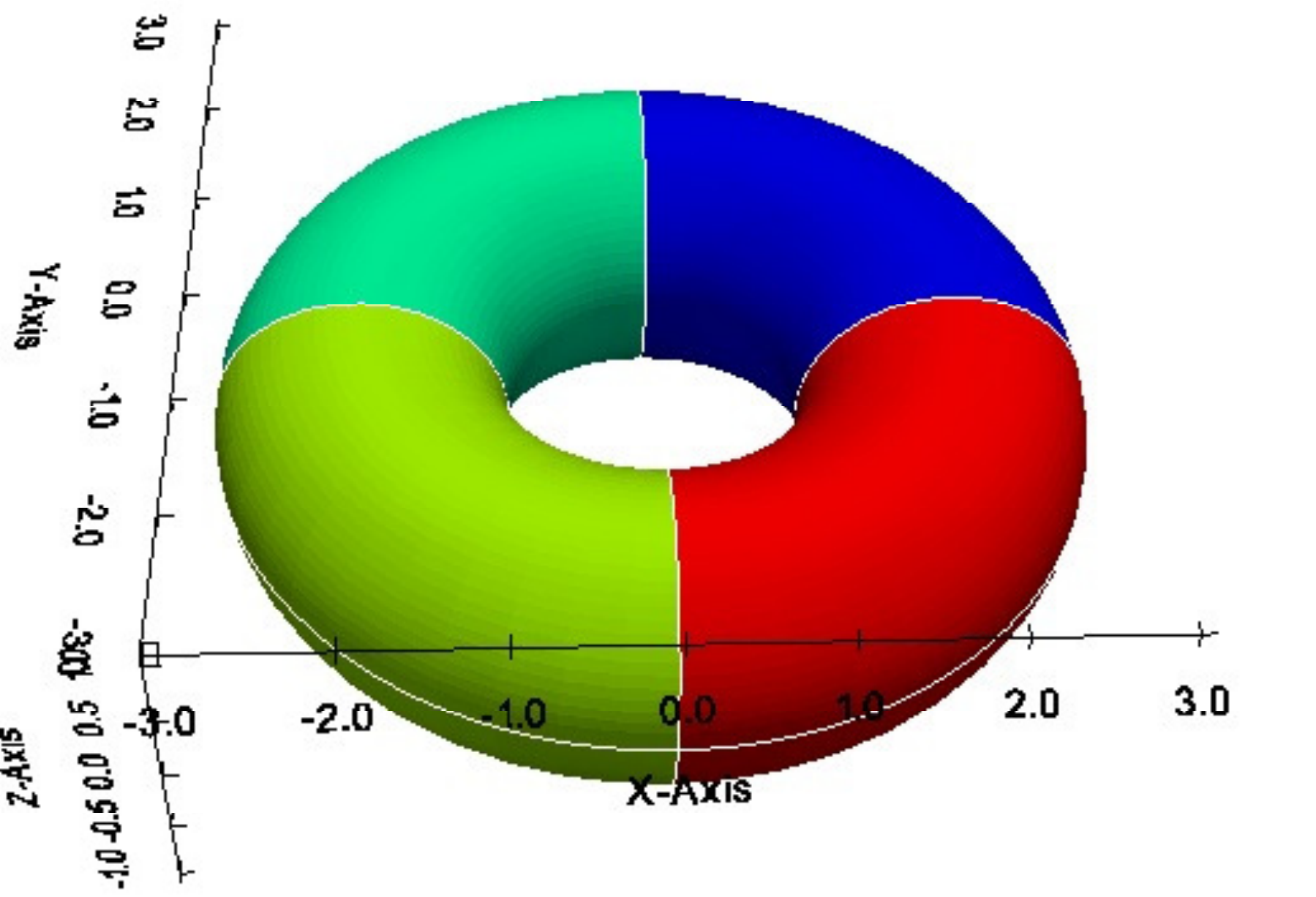}
      \caption{The computational domain consists of four (4) patches represented in different colours 
      $\Omega_i, i=1,\ldots,4$ with corresponding diffusion coefficient 
      $\alpha_i \in \{10^{-4}, 10^{4},10^{-4}, 10^{4}\}.$}
      \label{fig:computationaldomains}
 \end{figure}
\subsection{Open Surface}
\label{subsec:opensurface}
We consider a diffusion problem with homogeneous Dirichlet boundary condition on 
an open surface $\Omega$ that is given by a quarter cylinder in the first 
quadrant i.e. $x \geq 0$ and $y \geq 0$ with unitary radius and height $L = 4.$
The computational domain $\Omega$ is decomposed into 4 patches, with each of 
the patches having height of $L = 1$ and depicted by different color as seen on the 
left-hand side of Fig.~\ref{fig:computationaldomains} (left). The knot vectors 
representing the geometry of each patch are given by $\Xi_{1}= \{0, 0, 0, 1, 1, 
1\}$ and $\Xi_{2}= \{0, 0, 1, 1\}$ in the $\xi_1-$direction and  
$\xi_2-$direction respectively.
Let $f(\phi,z) = \varrho \left( \frac{\sigma^2 \pi^2}{L^2}g_{\phi,1}(\phi) - 
g_{\phi,2}(\phi) \right)g_{z}(z),$ where $\phi:= \arctan \left(\frac{x}{y} 
\right),$ $g_{\phi,1}(\phi):= (1- \cos(\phi)) (1-\sin(\phi)),$ 
$g_{\phi,2}(\phi) := (\cos(\phi) + \sin(\phi) -4\sin(\phi)\cos(\phi)),$
and $g_z(z) := \sin \left(\sigma \pi \frac{z}{L} \right)$
for $\sigma \in \mathbb{N}_0$ and $\varrho > 0.$ The exact solution of the 
problem is $u(\phi,z) = \varrho g_{\phi,1}(\phi)g_z(z).$ In our numerical
experiments, we set $\sigma = 3, \varrho = 1/\left(3/2 -\sqrt{2} \right).$
Fig.~\ref{fig:quartercylindergeometry}. 
We present the convergence behavior of the dGIGA scheme with respect to the 
discrete norm $\|\cdot\|_h$ in Fig..~\ref{fig:quartercylindergeometry} by 
successive mesh refinement of ratio $h_i/h_j = 2^q,$ where $q = 1,2,3$ and 
$q=4$ are the refinement level using NURBS of degree $p=2$ and $p =4.$ 
We observe the optimal convergence rate as theoretically predicted 
in \thmref{thm:discretenormerrorestimate} for smooth functions. 
 \begin{figure}[thb!]
  \centering
      \includegraphics[width=0.45\textwidth]{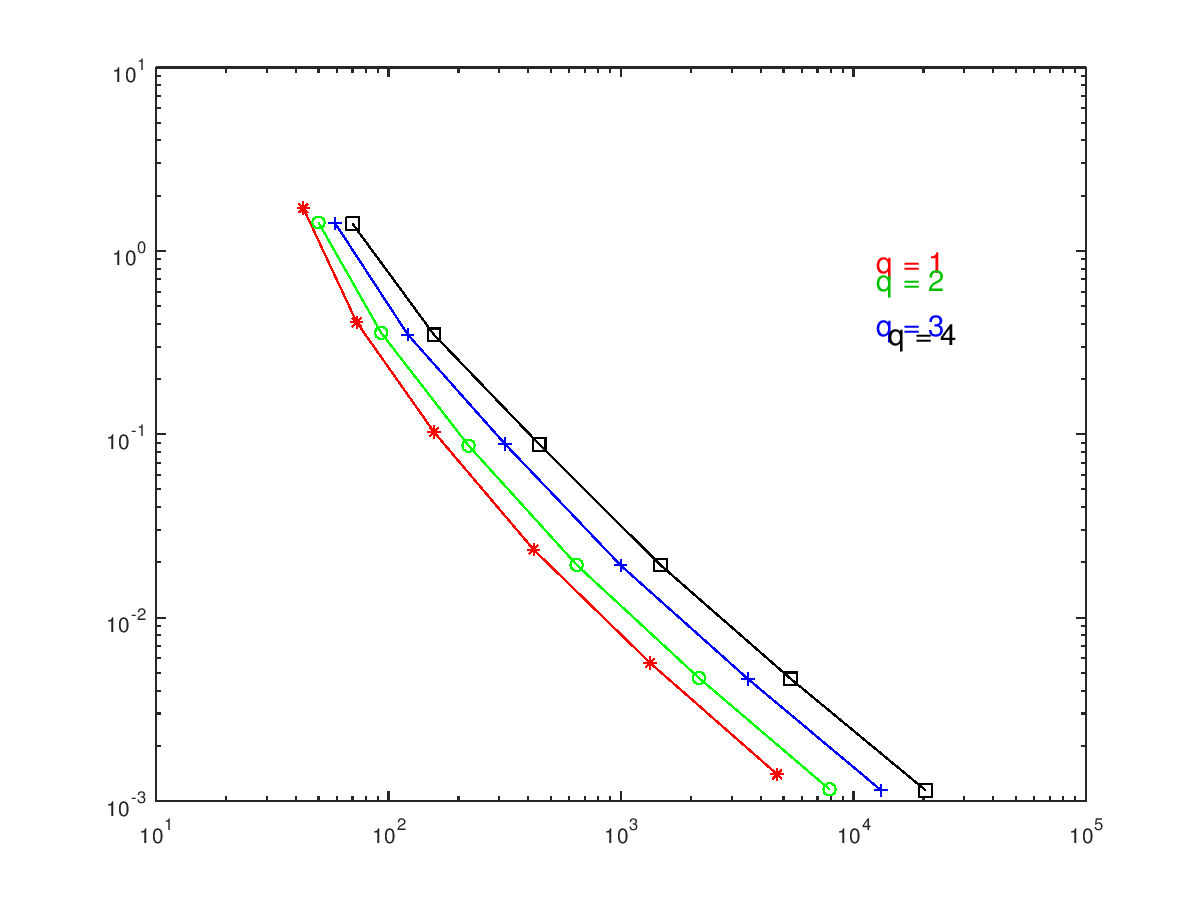}
      \includegraphics[width=0.45\textwidth]{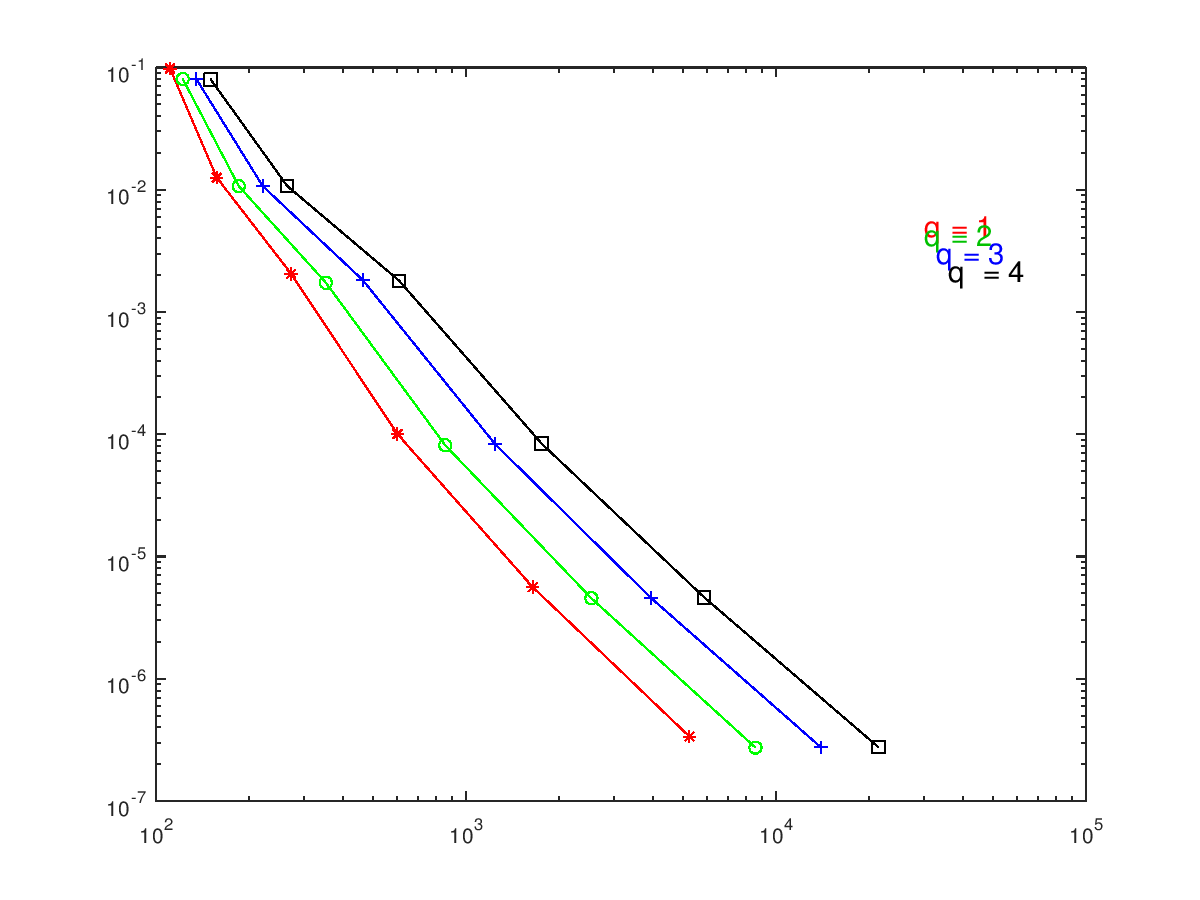}
      \caption{The convergence rate mesh refinement levels $q = 1,2,3$ and $q = 4$ 
      using B-spline degrees $p = 2$ (left) and $p = 4$ (right) for quarter-cylinder.}
      \label{fig:quartercylindergeometry}
 \end{figure}
\subsection{Closed Surface}
\label{subsec:closedsurface}
We consider the closed surface 
$$\Omega = \{(x,y) \in (-3,3)^2, z \in (-1,1): \; r^2 = z^2 + (\sqrt{x^2+y^2} 
-R^2) \},$$ 
that is nothing but a torus decomposed into 4 patches, see 
Fig.~\ref{fig:computationaldomains} (right). The knot vectors describing the 
NURBS used for the 
geometrical representation of the patches
$\Xi_{1}  = \{0, 0, 0, 0.25, 0.25, 0.50, 0.50, 0.75, 0.75, 1, 1, 1\}$
and $\Xi_{2}= \{0, 0, 0, 1, 1, 1\}$ . Let us consider the surface Poisson 
equation with the right-hand side 
\begin{align*}
f(\phi, \theta) &= r^{-2} \left(9\sin(3\phi)\cos(3\theta+\phi)\right) \\
&- \left((R + r\cos(\theta))^{-2} (-10\sin(3\phi)\cos(3\theta+\phi) - 
6\cos(3\phi)\sin(3\theta+\phi)) \right)\\
&- \left( (r(R+r\cos(\theta))^{-1}) (3\sin(\theta)\sin(3\phi)\sin(3\theta+\phi)) 
\right),
\end{align*}
where $\phi = \arctan (y/x)$, $\theta = \arctan (z/(\sqrt{x^2+y^2}-R))$, $R = 2$
and $r=1$.
The exact solution is given by $u= \sin(3 \phi)\cos(3\theta + \phi)$. The 
functions $u$ and $f$ are chosen  such that the zero mean compatibility 
condition holds.
We present the convergence behavior of the dGIGA scheme with respect to the 
discrete norm $\|\cdot\|_h$  by successive mesh refinement of ratio 
$h_i/h_j = 2^q,$ where $q = 1,2,3$ and $q=4$ are the number of mesh refinements 
using NURBS degrees $p=2$ and $p= 4$ see Fig.~\ref{fig:torusgeometry}.
We observe the optimal convergence rate as theoretically predicted in 
\thmref{thm:discretenormerrorestimate} for smooth functions. 
  \begin{figure}[th!]
  \centering
      \includegraphics[width=0.45\textwidth]{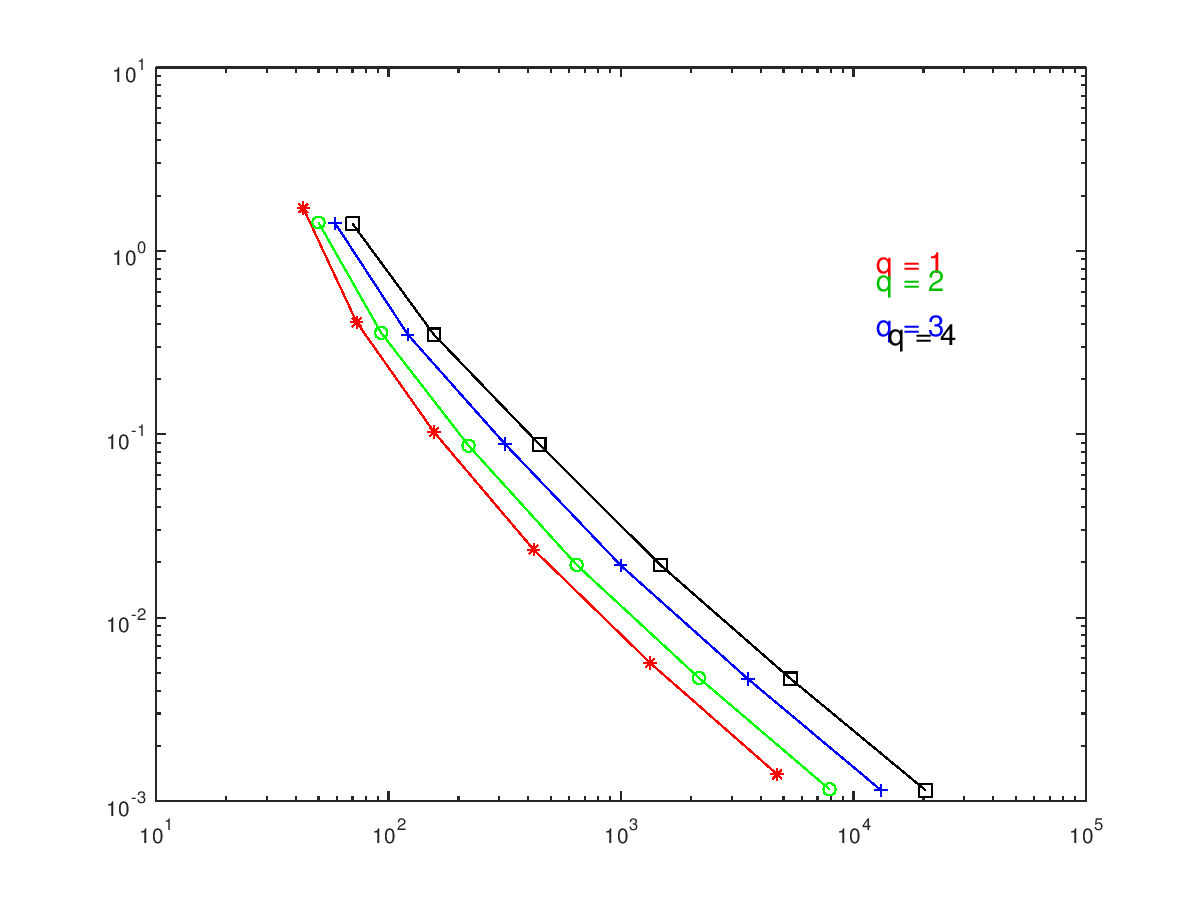}
      \includegraphics[width=0.45\textwidth]{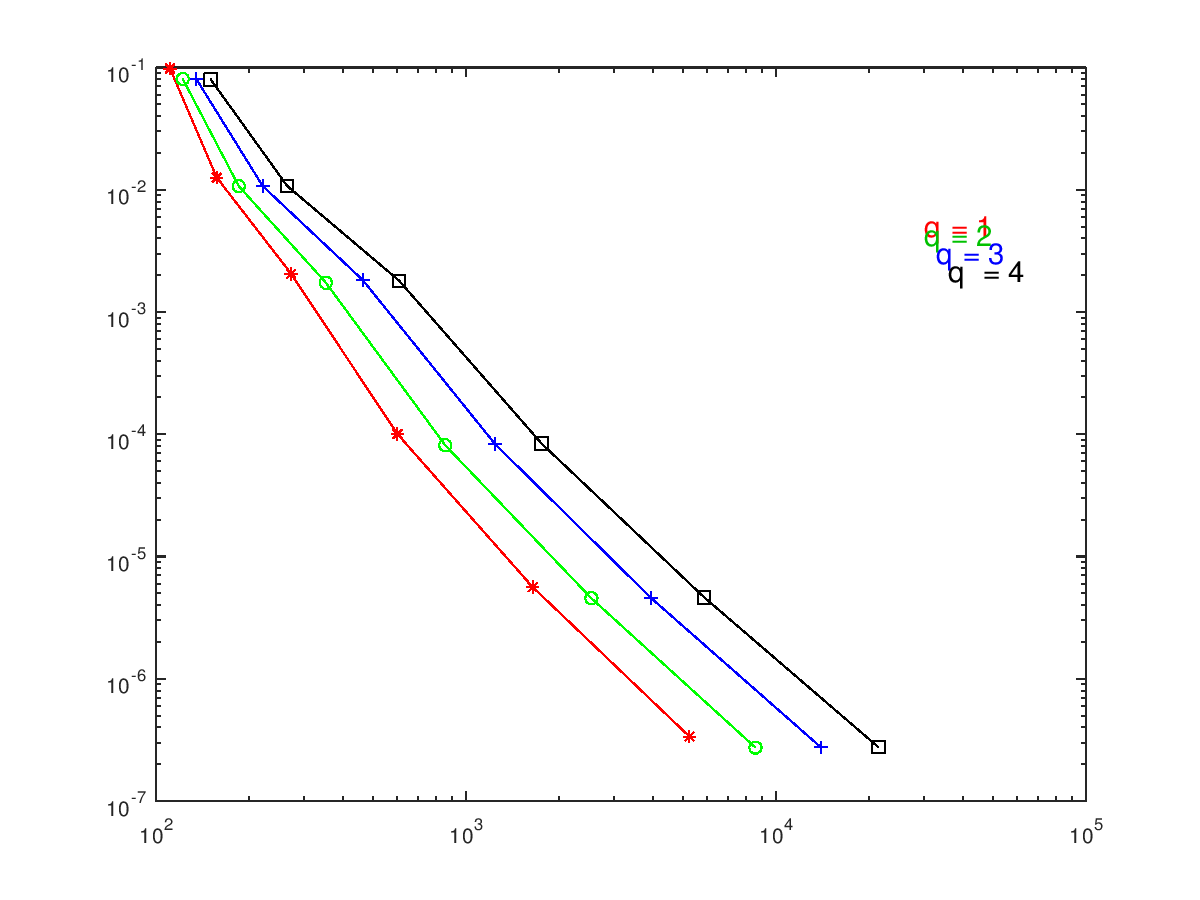}
      \caption{The convergence rate mesh refinement levels $q = 1,2,3$ and $q = 
4$ using B-spline degree $p = 2$ (left) and $p = 4$ (right).}
      \label{fig:torusgeometry}
 \end{figure}

\section*{Conclusion}
In this article, we considered the discontinuous Galerkin isogeometric analysis 
(dGIGA) for the surface diffusion problem with jumping coefficient and geometrically 
non-matching meshes. We analyzed the well-posedness and presented \textit{a 
priori} error estimates. Finally, we presented numerical results confirming the 
theory presented. In solving the linear system arising from the dGIGA scheme, we 
applied the dual-primal discontinuous Galerkin isogeometric tearing and 
interconnecting method (dG-IETI-DP). This involved a Preconditioned Conjugate 
Gradient (PCG) algorithm with the scaled Dirichlet preconditioner which is known to be 
robust with respect to jumping diffusion coefficient. An 
extension of the results to non-orientable surfaces as well evolving surfaces 
will be considered in our next article.
%
\section*{Acknowledgement}
The author acknowledges the Horizon 2020 Programme (2014-2020) under grant 
agreement number 678727.
\bibliographystyle{plain}
\bibliography{dGIGASurfacePoisson}
\end{document}